\documentclass[10pt,namelimits,sumlimits]{ip-journal}
\usepackage{amssymb,amsmath}
\usepackage[mathscr]{eucal}
\usepackage{color}
\usepackage{enumitem}
\usepackage[utf8]{inputenc}
\usepackage[leftcaption]{sidecap}
\usepackage{tikz}
 \usepackage{tikz-cd}
\usetikzlibrary{matrix,arrows,decorations.pathmorphing}
\usetikzlibrary{positioning}
\usetikzlibrary{matrix,arrows,decorations.pathmorphing, shapes.geometric}

\usepackage{psfrag}

\usepackage{epstopdf}
\usepackage{graphicx}
\usepackage{verbatim}
\usepackage{amsmath,amsthm}
\usepackage{graphics}
\usepackage{color}
\usepackage{epsfig}
\usepackage{fullpage}
\usepackage{amssymb,amsmath}
\usepackage[mathscr]{eucal}
\usepackage{ upgreek }

\theoremstyle{plain}
  \newtheorem{theorem}{Theorem}
    \newtheorem{corollary}{Corollary}
  \newtheorem{lemma}{Lemma}
\theoremstyle{definition}
  \newtheorem{definition}[subsection]{Definition}
  \newtheorem{remark}[subsection]{Remark}

\newtheorem{conjecture}{Conjecture}

\newcommand{\R}{\mathbb{R}}

\newcommand{\N}{\mathbb{N}}

\newcommand{\Ncal}{\mathcal{N}}
\newcommand{\Ccal}{\mathcal{C}}
\newcommand{\Acal}{\mathcal{A}}
\newcommand{\Scal}{\mathcal{S}}
\newcommand{\Gcal}{\mathcal{G}}

\begin{document}

\title[]{A Characterization of the Critical Catenoid}
\author[P.~McGrath]{Peter~McGrath}

\date{}
\address{Department of Mathematics, University of Pennsylvania, Philadelphia,
PA 19104}  \email{pjmcgrat@sas.upenn.edu}
\maketitle
\begin{abstract}
We show that an embedded minimal annulus $\Sigma^2 \subset B^3$ which intersects $\partial B^3$ orthogonally and is invariant under reflection through the coordinate planes is the critical catenoid.  The proof uses nodal domain arguments and a characterization, due to Fraser and Schoen, of the critical catenoid as the unique free boundary minimal annulus in $B^n$ with lowest Steklov eigenvalue equal to 1.  We also give more general criteria which imply that a free boundary minimal surface in $B^3$ invariant under a group of reflections has lowest Steklov eigenvalue 1.

\end{abstract}

\section{Introduction}
\label{Sintro}

A classical theorem of Nitsche \cite{Nitsche} states that the only free boundary minimal disks in the unit ball $B^3$ are flat equators.  It is not known whether or not similar rigidity theorems hold for surfaces with larger genus or number of boundary components, but in this direction and Fraser and Li  \cite{FL} made the following conjecture:

\begin{conjecture}[Fraser and Li, \cite{FL}]
\label{conj: FL}
Up to congruences, the critical catenoid is the unique properly embedded free boundary minimal annulus in $B^3$.  
\end{conjecture}

The study of free boundary minimal surfaces in $B^3$ has features in common with the study of closed minimal surfaces in $S^3$ (see \cite{Brendle:survey} for a recent survey on the latter topic).  A natural  counterpart to Conjecture \ref{conj: FL} in the setting of embedded minimal surfaces in $S^3$ - Lawson's conjecture that the Clifford torus is the only embedded minimal torus in $S^3$, up to congruences - was recently and spectacularly resolved by Brendle \cite{Brendle:lawson}.  Brendle's proof of the Lawson conjecture gives some reason to believe that Conjecture \ref{conj: FL} may be true.  One purpose of this note is to prove the following. 
\begin{theorem}
\label{Tmain2}
Let $\Sigma \subset B^n$, $n\geq 3$ be an embedded free boundary minimal annulus.  Suppose $\Sigma$ is invariant under reflection through three orthogonal hyperplanes $\Pi_i, i=1,2,3$ and $\partial\Sigma \cap \left( B^n\setminus  \bigcup_{i=1}^3 \Pi_i\right) \neq \emptyset$.  Then $\Sigma$ is congruent to the critical catenoid.  
\end{theorem}

As a consequence, we confirm Conjecture \ref{conj: FL} under a stronger hypothesis: 

\begin{corollary}
\label{Tmain}
Let $\Sigma \subset B^3$ be a embedded free boundary minimal annulus, symmetric with respect to the coordinate planes.  Then $\Sigma$ is congruent to the critical catenoid.  
\end{corollary}

Before Brendle's proof, earlier work of Ros established that the conclusion of the Lawson conjecture held with an additional strong assumption of symmetry with respect to the coordinate planes.  In particular, Ros proved (Theorem 6, \cite{Ros}): 

\begin{theorem}[Ros, \cite{Ros}]
\label{Tros}
Let $M \subset S^3$ be an embedded minimal torus, symmetric with respect to the coordinate hyperplanes of $\R^4$.  Then $M$ is the Clifford torus.  
\end{theorem}

As Conjecture \ref{conj: FL} remains open, we present Theorem \ref{Tmain2} and Corollary \ref{Tmain} as a natural step towards its resolution.  Note that
Corollary \ref{Tmain} follows directly from Theorem \ref{Tmain2}: if $\Sigma \subset B^3$ is as in Theorem \ref{Tmain}, it must be that $\partial \Sigma \cap \left(B^3 \setminus \bigcup_{i=1}^3 \Pi_i\right) \neq \emptyset$,  since otherwise the convex hull property would imply $\Sigma \subset \Pi_i$ for some $i\in \{1, 2, 3\}$, which would imply $\Sigma$ is a disk. In fact, the methods used to prove Theorem \ref{Tmain2} also lead to a new proof of Ros' theorem.

Fraser and Schoen have developed a fruitful connection between free boundary minimal surfaces in $B^n$ and the so-called Steklov eigenvalue problem.  One consequence of their work is the following characterization of the critical catenoid (Theorem 6.6, \cite{FS}):
\begin{theorem}[Fraser-Schoen, \cite{FS}]
\label{TFS}
Suppose $\Sigma$ is a free boundary annulus in $B^n$ such that the coordinate functions are first Steklov eigenfunctions.  Then $n=3$ and $\Sigma$ is congruent to the critical catenoid.  
\end{theorem}

In light of this result, to prove Theorem \ref{Tmain2}, it suffices to prove that the lowest Steklov eigenvalue $\sigma_1(\Sigma)$ on a free boundary minimal annulus $\Sigma$ with the assumed symmetries is $1$.  For this, we adapt a symmetrization technique of Choe and Soret \cite{Choe}, who proved that the lowest eigenvalue of the Laplacian on the Lawson and Karcher-Pinkall-Sterling minimal surfaces in $S^3$ is $2$.  A similar technique was used to estimate the first Steklov eigenvalue for domains in the plane with reflectional symmetries \cite{Proc}.  

More generally, we use these techniques to give general conditions on a free boundary minimal surface $\Sigma \subset B^3$ which guarantee that $\Sigma$ has lowest Steklov eigenvalue equal to $1$.  

Fraser and Schoen \cite{FS} have constructed a family $\{M_n\}_{n\in \N}$ of free boundary minimal surfaces with $\sigma_1(M_n) = 1$  which are characterized by the fact that they maximize $\sigma_1(\Sigma) |\partial \Sigma|$ over all surfaces $\Sigma$ with genus zero and $n$ boundary components.  These surfaces are expected to have reflectional symmetries, and recently Folha, Pacard, and Zolotareva \cite{Z} have described a family of genus $0$ and a family of genus $1$ free boundary minimal surfaces with a large number of boundary components which \emph{are} invariant under a group of reflections.  The genus $0$ family is expected to coincide with the family $\{M_n\}$ when $n$ is large; our methods show that any surfaces with these types of symmetries have lowest Steklov eigenvalue $1$.  This supports a conjecture of Fraser and Li \cite{FL} that $\sigma_1(\Sigma) = 1$ for any embedded free boundary minimal surface $\Sigma \subset B^3$. 

In Section \ref{Ssteklov}, we summarize the connection between free boundary minimal surfaces in $B^n$ and the Steklov eigenvalue problem and review some facts about nodal domains.  In Section \ref{Sproof}, we prove Theorem \ref{Tmain2} and indicate another proof of Ros' theorem.  Finally, in Section \ref{Sgeneral}, we investigate free boundary minimal surfaces $\Sigma \subset B^3$ with more general reflectional symmetries.  

The author thanks Brian Freidin and David Wiygul for helpful discussions, and Alex Mramor for comments on an earlier draft of this paper. Thanks are also due to the Mathematics Department of the University of Calfornia, Irvine for its generous hospitality during the time this work was carried out, and to Richard Schoen, whose inspiring lectures precipitated this work.

%
%
 
\section{The Steklov Eigenvalue Problem and Free Boundary Minimal Surfaces in $B^n$}
\label{Ssteklov}
Let $(\Sigma^n, \partial \Sigma, g)$ be a smooth compact Riemannian manifold with boundary.  Let $\eta$ be the unit outward pointing conormal vector field on $\partial \Sigma$.  The \emph{Steklov Eigenvalue Problem} is
\begin{equation}
\label{eqn: steklov}
  \begin{cases} 
      \hfill \Delta u  = 0    \hfill & \text{in}\quad \Sigma \\
      \hfill \frac{\partial u}{\partial \eta} = \sigma u \hfill & \text{on}\quad \partial \Sigma \\
  \end{cases}
\end{equation}
and we call a function $u$ satisfying \eqref{eqn: steklov} a \emph{Steklov eigenfunction}.  
The eigenvalues of \eqref{eqn: steklov} are the spectrum of the \emph{Dirichlet to Neumann map} $L: C^\infty(\partial \Sigma) \rightarrow C^{\infty}(\partial \Sigma)$ given by
\[ L u = \frac{\partial \hat{u}}{d\eta} \]
where $\hat{u}$ is the harmonic extension of $u$ to $\Sigma$ (that is, $\Delta \hat u = 0$ in $\Sigma$ and $\hat{u} = u$ on $\partial \Sigma$).
It is well known that $L$ is a self-adjoint pseudodifferential operator with discrete spectrum
\[ 0 = \sigma_0 < \sigma_1 \leq \sigma_2 \leq \dots. \]
Furthermore, the first nonzero eigenvalue $\sigma_1$ may be defined variationally by
\begin{align}
\label{Esigma1}
\sigma_1 = \inf_{u\in \Ccal} \frac{ \int_\Sigma |\nabla u |^2 }{\int_{\partial \Sigma}u^2},
\end{align}
where
\[ \Ccal = \{ u \in C^0(\Sigma): u \not\equiv 0,  \int_{\partial \Sigma} u = 0, \text{ and $u$ is piecewise $C^1$}\} .\] 
If $u \in \Ccal$ achieves equality in the Rayleigh quotient \eqref{Esigma1}, then $u$ is an eigenfunction with eigenvalue $\sigma_1$.  Hereafter, we refer to such $u$ as first eigenfunctions.  See \cite{GP} for more information about the Steklov eigenvalue problem.    
\begin{definition}
\label{Dnodal}
Suppose $u$ is a Steklov eigenfunction.  The \emph{nodal set of $u$} is 
\[ \Ncal = \{ p \in \Sigma : u(p) = 0\}.\]
A \emph{nodal domain of $u$} is a connected component of $\Sigma \setminus \Ncal$.  
\end{definition}
It is a standard fact (cf. \cite{Cheng}) that when $n=2$, $\Ncal$ consists of finitely many $C^1$ arcs which intersect at a finite set of points.  We call a piecewise $C^1$ curve in $\Ncal$ a \emph{nodal line}.  The following Courant nodal domain type theorem is standard (see \cite{Cheng}, \cite{Proc}).  We provide a proof for the sake of completeness.
\begin{lemma}
\label{Lcourant}
If $u$ is a first eigenfunction, then $u$ has exactly two nodal domains. 
\end{lemma}
\begin{proof}
Since $\int_{\partial \Sigma} u = 0$ and $u \not \equiv 0$, $u$ has at least two nodal domains.  Suppose for a contradiction that $u$ is a first eigenfunction and $\Omega_1, \Omega_2, \Omega_3$ are distinct, nonempty nodal domains.  We first claim that $\partial \Sigma \cap \partial \Omega_i \neq \emptyset$ for $i=1,2,3$.  To see this, observe that if there were an $i$ for which $\partial \Sigma \cap \partial \Omega_i = \emptyset$, then the maximum principle would imply $u\equiv 0$ on $\Omega_i$, and the unique continuation principle for elliptic equations \cite{A} would imply that $u\equiv 0$ on $\Sigma$, contrary to our assumption.

Therefore, there exist nonzero $c_1, c_2 \in \R$ such that the function $v$ defined by
\begin{equation}
\label{eqn: test}
v = 
  \begin{cases} 
      \hfill c_1 u  \hfill & \text{on}\quad \Omega_1 \\
      \hfill c_2 u  \hfill & \text{on}\quad \Omega_2\\
      \hfill 0 \hfill & \text{on} \quad \Sigma\setminus (\Omega_1 \cup \Omega_2)
        \end{cases}
\end{equation}
satisfies $\int_{\partial \Sigma} v = 0$.  Then by construction, $v \in \Ccal$.  Moreover, since $\frac{\partial v}{\partial \eta} = \sigma_1 v$ on $\partial \Sigma$, integration by parts yields
\begin{align}
\label{Eparts}
\int_{\partial \Sigma} \sigma_1 v^2 = \int_{\partial \Sigma} v \frac{\partial v}{\partial \eta} 
= \int_\Sigma \left| \nabla v\right|^2,
\end{align}

so $v$ is a first eigenfunction by \eqref{Esigma1}.  In particular, $\Delta v = 0$ on $\Sigma$.  On the other hand, $v$ vanishes on a nonempty open set by \eqref{eqn: test}, so the unique continuation property implies that $v\equiv 0$ on $\Sigma$, which contradicts our hypothesis. 
\end{proof}

\begin{definition}
\label{Dfbms}
We say the image of a minimal immersion $F: \Sigma \rightarrow B^n$ is a \emph{free boundary minimal surface} if $F(\partial \Sigma) \subset \partial B^n$ and $F(\partial \Sigma)$ intersects $\partial B^n$ orthogonally. 
\end{definition}
Abusing notation slightly by identifying $\Sigma$ with its image under $F$, the free boundary condition amounts to the condition that $\eta = X$ along $\partial \Sigma$, where $X$ is the position vector in $\R^n$.  Since the coordinate functions of any minimal immersion $F: \Sigma \rightarrow \R^n$ are harmonic and the free boundary condition implies $\frac{\partial x_i}{\partial \eta} = x_i$, we have the following.   
\begin{remark}
\label{Pcoord}
If $\Sigma \subset B^n$ is a free boundary minimal surface, the coordinate functions $x_1, \dots, x_n$ are Steklov eigenfunctions with eigenvalue $1$.  
\end{remark}

The \emph{critical catenoid} is a particular scaling of the catenoid in $\R^3$ that satisfies the free boundary condition.  It may be parametrized by $F_{cat}: [-\rho_0, \rho_0]\times [0, 2\pi]\rightarrow \R^3$, where 
\[ F_{cat}(r, \theta) = \frac{1}{\rho_0 \cosh \rho_0}( \cosh r \cos \theta, \cosh r \sin \theta, r)\]
and $\rho_0$ is defined to be the solution of the equation $\rho_0\tanh \rho_0 = 1$.

%
%
\section{Proof of Theorem \ref{Tmain}}
\label{Sproof}

Let $\Pi \subset \R^n$ be a hyperplane and let $R_\Pi: \R^n\rightarrow \R^n$ be the orthogonal reflection through $\Pi$. We say a surface $\Sigma$ is invariant under reflection through $\Pi$, or $R_\Pi$-invariant, if $R_\Pi(\Sigma) = \Sigma$.  
\begin{definition}
\label{Dsym}
Suppose $\Sigma$ is $R_\Pi$-invariant.  We define operators $\Acal_{\Pi}, \Scal_{\Pi} : C^0(\Sigma) \rightarrow C^0(\Sigma)$ by
\begin{align}
\label{Esym}
\Acal_\Pi u  = \frac{1}{2}\left(u - u \circ R_\Pi\right), \quad \Scal_\Pi u = \frac{1}{2}\left(u + u \circ R_\Pi\right).
\end{align}
\end{definition}

Clearly $\Acal_\Pi u$ and $\Scal_\Pi u$ are the antisymmetric and symmetric parts of $u$ about $\Pi$, i.e., $u = \Acal_{\Pi}  u + \Scal_{\Pi} u $ and
\[  \Acal_{\Pi}\circ R_\Pi  u = - \Acal_{\Pi}u,\quad \quad  \Scal_{\Pi}\circ R_\Pi  u = \Scal_\Pi u.\] 

If $u$ is a first eigenfunction, linearity of \eqref{eqn: steklov} and \eqref{Esym} implies $\Acal_\Pi u$ and $\Scal_\Pi u$ are eigenfunctions, in fact first eigenfunctions unless $\Acal_{\Pi} u \equiv 0$ or $\Scal_\Pi u \equiv 0$, respectively. 

\begin{lemma}
\label{Lsym}
If $\Sigma^2 \subset \R^n$ is $R_\Pi$-invariant,  $\sigma_1 < 1$, and $u$ is a first eigenfunction, then $u=\Scal_\Pi u$.  
\end{lemma}
\begin{proof}
Suppose for a contradiction that $\Acal_\Pi u\not\equiv 0$.  Then $\Acal_\Pi u$ is a first eigenfunction. By Lemma \ref{Lcourant}, $\Acal_\Pi u$ has exactly two nodal domains; let $\Omega$ be the nodal domain where $u>0$.  Since $\Acal_\Pi u$ is antisymmetric, its nodal set $\Ncal$ contains $\Sigma \cap \Pi$ and the second nodal domain is $R_\Pi(\Omega)$.  Since $\Pi$ disconnects $\Sigma$, $\Omega$ lies on one side of $\Pi$. 

Define $\varphi : = \langle X, w\rangle$, where $X\in \R^n$ is the position vector field and $w\in \R^n$ is a normal vector to $\Pi$ such that $\Omega \subset \{ \varphi>0\}$.      
 It follows that $\int_{\Omega} \varphi \Acal_\Pi u>0$.  Since $\Acal_\Pi \varphi = \varphi$, we have then
\begin{align}
\label{Eprod}
\int_{\partial \Sigma} \varphi  \Acal_\Pi u  &= \int_{\partial \Sigma \cap \Omega} \varphi \Acal_\Pi u  + \int_{\partial \Sigma \cap R_\Pi (\Omega)} \Acal_\Pi\varphi \, \Acal_\Pi u   = 2\int_{\partial \Sigma \cap\Omega} \varphi \Acal_\Pi u >0.
\end{align} 
On the other hand, Remark \ref{Pcoord} implies that $\varphi$ is an eigenfunction with eigenvalue $1$.  Since $L$ is self-adjoint and $\sigma_1<1$, we have 
$
\int_{\partial \Sigma}  \varphi  \Acal_\Pi u= 0, 
$
which contradicts \eqref{Eprod}.
\end{proof}

\begin{proof}[Proof of Theorem \ref{Tmain2}]
Let $\Sigma^2 \subset B^n$ be an embedded, $\Gcal$-invariant free boundary annulus, where $\Gcal = \langle R_{\Pi_1}, R_{\Pi_2}, R_{\Pi_3}\rangle$.  Without loss of generality, we may assume $\Pi_i$, $i=1,2,3$ are the first three coordinate planes.  By Theorem \ref{TFS}, it suffices to prove $\sigma_1(\Sigma) = 1$.  For the sake of a contradiction, suppose $\sigma_1(\Sigma)<1$ and let $u$ be a first eigenfunction.  

Since $\partial \Sigma \cap \left( B^n\setminus  \bigcup_{i=1}^3 \Pi_i\right) \neq \emptyset$, it follows from the symmetries that $\partial \Sigma \cap \text{int}(O) \neq \emptyset$, where $O = \{ (x_1, \dots, x_n) \in \R^n: x_1, x_2, x_3 \geq 0\}$. Let $\Gamma\subset \partial B^n$ be a boundary component of $\partial \Sigma$.  Since $\Sigma$ is embedded, any intersection of $\Gamma$ with one of the planes $\Pi_i$ is orthogonal. Since $\Sigma$ is an annulus, the symmetries imply there is at least one such intersection.  

Now consider the curve $\gamma := \Gamma \cap O$.  Since $\Gamma$ is an embedded $\Gcal$-invariant circle, $\gamma$ is connected, embedded and intersects exactly two (at first possibly indistinct) coordinate planes $\Pi_i$, $i=1,2,3$.  We parametrize $\gamma$ on $[0, 1]$ and may suppose without loss of generality that $\gamma(0) \in \Pi_1$.   Since $\Sigma$ is an annulus, the symmetries imply that $\gamma(1) \notin \Pi_1$.  Without loss of generality, we can assume $\gamma(1) \in \Pi_2$.  By the symmetries, the other boundary component of $\partial \Sigma$ is $R_{\Pi_3}(\Gamma)$.  Since $\Sigma$ is an annulus, this implies $\Sigma\cap \Pi_3$ is a circle, and moreover that $D: = \Sigma\cap O$ is bounded by four arcs $\gamma, e_1, e_2, e_3$, with $e_i \subset \Sigma\cap \Pi_i$, $i=1,2,3$.  In particular, $D$ is a fundamental domain for $\Sigma$.  

We first claim that $\Ncal$ contains a nodal line $\ell$ which intersects $\gamma$ at an interior point.  If not, then $u$ would not change sign on $\gamma$, and since $u$ is $\Gcal$-invariant and $D$ is a fundamental domain for $\Sigma$, we would have $\int_{\partial \Sigma} u \neq 0$, which contradicts \eqref{Esigma1}.  There are four cases, depending on whether the second intersection of $\ell$ with $\partial D$ is on $\gamma, e_1, e_2$, or $e_3$.
 
By Lemma \ref{Lsym}, $\Ncal$ is $\Gcal$-invariant.  If $\ell$ ends on $\gamma$,  $\Gcal$-invariance of $\Ncal$ implies $u$ has at least 9 nodal domains.
If $\ell$ ends on $e_1$,  $\Gcal$-invariance of $\Ncal$ implies $u$ has at least 5 nodal domains. In the same fashion, $u$ has at least 5 nodal domains if $\ell$ ends on $e_2$.  
Finally, if $\ell$ ends on $e_3$, $\Gcal$-invariance of $\Ncal$ implies $u$ has at least 4 nodal domains. 

Each of these cases contradicts Lemma \ref{Lcourant}, so it must be that $\sigma_1 = 1$.
\end{proof}

Similar reflection techniques were developed in \cite{Choe} in the setting of closed embedded minimal surfaces in the round three-sphere $S^3$.  These arguments together with ideas from the proof of Theorem \ref{Tmain2} above yield a new proof of a result of Ros (Theorem \ref{Tros} above).
\begin{proof}[Proof of Theorem \ref{Tros}]
Suppose $M\subset S^3$ is an embedded minimal $\Gcal$-invariant torus, where $\Gcal = \langle \Pi_1, \Pi_2, \Pi_3, \Pi_4\rangle$ and $\Pi_i, i=1, \dots, 4$ are the coordinate planes in $\R^4$.  Then $\Pi_4$ divides $M$ into two embedded annuli which intersect the two-sphere $\{(x_1, x_2, x_3, x_4) \in S^3: x_4 = 0\}$ orthogonally and are congruent by the reflection $R_{\Pi_4}$.  Stereographic projection from $(0, 0, 0, 1)$ gives a conformal identification 
\[ \{ (x_1, x_2, x_3, x_4) \in S^3: x_4\leq 0\} \cong B^3.\]
  Moreover, under this equivalence, $M\cap \{ x_4 \leq 0\}$ maps to an embedded annulus $\Sigma$ which is $R_{\Pi_i}$-invariant, $i=1, 2,3$ and intersects $\partial B^3$ orthogonally.  By the topological argument in the proof of Theorem \ref{Tmain} above, $\Sigma \cap O$ is bounded by four arcs and is in particular simply connected.  

Therefore, the inverse image of $\Sigma \cap O$ under stereographic projection is a simply connected fundamental domain for $M$.  We may then use Theorem 2 from \cite{Choe} to conclude the first eigenvalue $\lambda_1$ of the laplacian on $M$ is $2$.  By a result of Montiel and Ros (\cite{MR}, Theorem 4)  this implies $M$ is congruent to the Clifford torus.  
\end{proof}

%
%
\section{Surfaces $\Sigma$ with General Reflectional Symmetries}
\label{Sgeneral}
The results of this section should be compared with those of Section 4 of \cite{Choe}, where analogous results are discussed for the first eigenvalue of the Laplacian on embedded minimal surfaces in $S^3$.  
We assume now that $\Sigma \subset B^3$ is an embedded free boundary minimal surface invariant under a finite group of reflections $\Gcal$.  Moreover, we assume the pair $(\Sigma, \Gcal)$ satisfies the following conditions.  
\begin{enumerate}
\item[(C1)] The fundamental domain $W$ for $\Gcal$ in $B^3$ is a tetrahedral wedge bounded by planes $\Pi_i, i=1,2,3$ and $\partial B^3$.  
\item[(C2)] The fundamental domain $D: = W \cap \Sigma $ for $\Sigma$ intersects $\partial \Sigma$ in a single connected curve $\gamma$.  
\end{enumerate}
We do not require at first that $D$ is simply connected.  If $U\subset D$ and $U$ is a connected component of $\tilde{U}\cap D$, where $\tilde{U}$ is a nodal domain, we say $U$ is a \emph{nodal domain in $D$}.  

\begin{lemma}
\label{Ldomain}
Let $\Sigma \subset B^3$ be a $\Gcal$-invariant free boundary minimal surface satisfying (C1) and (C2) and suppose $\sigma_1(\Sigma)<1$.  Then $D$ has exactly two nodal domains $U$ and $V$, and each $\partial U$ and $\partial V$ has nonempty intersection with $\partial B^3$ and $\Pi_i$ for $i=1,2,3$.
\end{lemma}
\begin{proof}
Let $u$ be a first eigenfunction and let $\tilde{U}, \tilde{V}$ be the nodal domains for $u$ from  Lemma \ref{Lcourant}.  Because $\int_{\partial \Sigma} u =0$, $D$ is a fundamental domain for $\Sigma$, and $D \cap \partial \Sigma = \gamma$, $u$ changes sign on $\gamma$.  Therefore $u$ has at least two nodal domains in $D$.  

Suppose $\tilde{U}\cap D = U_1\cup U_2$ for some relatively open sets $U_1, U_2$ in $D$.  Let  $\tilde{\beta}: [0, 1]\rightarrow \tilde{U}$ be a path with $\tilde{\beta}(0) \in U_1$ and $\tilde{\beta}(1) \in U_2$.  Reflecting parts of $\tilde{\beta}$ by appropriate elements of $\Gcal$ gives rise to a path $\beta: [0, 1] \rightarrow  D$ with $\beta(0) \in U_1$ and $\beta(1) \in U_2$.  Therefore $U_1$ and $U_2$ are not disjoint and so $U: = D\cap \tilde{U}$ is connected.  The same argument shows $V: = D\cap \tilde{V}$ is connected,  hence $D$ has exactly two nodal domains, $U$ and $V$.  

Suppose for the sake of a contradiction that $\partial U \cap \Pi_i  = \emptyset$ for some $i\in \{1,2,3\}$.  By Lemma \ref{Lsym}, $R_{\Pi_i}(\tilde{U})$ is a nodal domain. Since $u$ takes different signs on $\tilde{U}$ and $\tilde{V}$, Lemma \ref{Lcourant} implies that $\tilde{U} =R_{\Pi_i}(\tilde{U})$.  Let $p\in \tilde{U}$.  Then there exists a continuous curve $\beta: [0, 1] \rightarrow \tilde{U}$ with $\beta(0) = p$ and $\beta(1) = R_{\Pi_i}(p)$.  

Let $H$ be the index two subgroup of $\Gcal$ naturally identified with the quotient $\Gcal / R_{\Pi_i}$.  In particular, $D \cup R_{\Pi_i}(D)$ is a fundamental domain for $H$.  After possibly reflecting parts of $\beta$ by elements of $H$, we may suppose that $\beta(t) \in \overline{U \cup R_{\Pi_i}(U)}$ for all $t\in [0, 1]$.  But $\Pi_i$ disconnects $\overline{U \cup R_{\Pi_i}(U)}$, which implies there is a $t_0\in (0, 1)$ such that $\gamma(t_0) \in \Sigma \cap \Pi_i \subset V$.  This a contradiction.  
\end{proof}

\begin{theorem}
\label{Tgeneral}
Suppose $\Sigma^2 \subset B^3$ is a $\Gcal$-invariant free boundary minimal surface satisfying (C1) and (C2).  Suppose $D$ is simply connected and $\partial D$ consists of at most $5$ edges.  Then $\sigma_1(\Sigma) =1$.  
\end{theorem}
\begin{proof}
Suppose on the contrary that $\Sigma$ is as above and $\sigma_1(\Sigma)<1$.  Then $\partial D$ decomposes into a curve $\gamma$ along $\partial B^3$ and a sequence of no more than four edges $e_j$, each of which is contained in one of the planes $\Pi_i$.  As in the proof of Theorem \ref{Tmain2}, $D$ contains a nodal line $\ell$ which intersects an interior point of $\gamma$.  Since $D$ is simply connected, $\ell$ divides $D$.
We next claim that $\Ncal = \ell$.  If not, $\Ncal$ contains a second nodal line $\ell'$ and either \newline
(i). $\ell'$ is a closed curve in $D$.\newline
(ii). $\ell'$ intersects $\partial D$ twice and is disjoint from $\ell$.\newline
(iii). $\ell'$ intersects $\ell$.\newline
Since $D$ is simply connected, each of the above cases implies $D$ has at least three nodal domains, which contradicts Lemma \ref{Ldomain}.  Therefore, $\ell = \Ncal$, and so $\ell$ divides $D$ into exactly two nodal domains $U$ and $V$.

Since $\partial D$ consists of at most four edges other than $\gamma$, either $U$ or $V$ is bounded by edges intersecting at most two of the three planes $\Pi_i$.  This contradicts Lemma \ref{Ldomain}.
\end{proof}

\begin{corollary}
\label{CZ}
The surfaces $\Sigma_n$ and $\tilde{\Sigma}_n$ described by Folha, Pacard and Zolotareva in \cite{Z} have first Steklov eigenvalue equal to $1$.  
\end{corollary}
\begin{proof}
Let $\{e_1, e_2, e_3\}$ be the standard orthonormal frame for $\R^3$. Define 
\begin{align*}
v_1 =  -\sin \left(\frac{\pi}{n}\right)e_1 + \cos\left( \frac{\pi}{n}\right) e_2,  \quad v_2= e_2, \quad v_3 = e_3 
\end{align*}
and $\Gcal_n = \langle R_{\Pi_{v_1}}, R_{\Pi_{v_2}}, R_{\Pi_{v_3}}\rangle$.
The surfaces in \cite{Z} are invariant under a group of isometries $\mathfrak{S}_n \subset O(3)$ (cf. the discussion above Theorem 1.1 in \cite{Z}), and it is straightforward to see that $\mathfrak{S}_n = \Gcal_n$.  

 In particular, the fundamental domain $D: = W\cap S_n$ is simply connected and consists of four edges $\gamma, e_1, e_2, e_3$ with $\gamma \in \partial B^3$ and $e_i \in \Pi_{v_i}, i=1,2,3$.  The fundamental domain $D'_n: = W\cap \Sigma_n$ is a small normal perturbation of $D_n$ and consequently $\partial D'_n$ consists of four edges $\gamma', e'_1, e'_2, e'_3$ with $\gamma' \in \partial B^3, e'_i\in \Pi_{v_i}, i=1,2,3$.  Thus, Theorem \ref{Tgeneral} implies $\sigma_1(\Sigma_n) = 1$.  

An analogous argument applies to the surfaces $\tilde{\Sigma}_n$, except that the approximately catenoidal neck centered at the origin means the fundamental domain $\tilde{D}'_n: = \tilde{\Sigma}_n\cap W$ consists of five edges $\gamma, e_i', i=1,\dots, 4$ with $e'_i\in \Pi_{v_i}, i=1,2,3$ and $e_4 \in \Pi_{v_3}$.  In any case, we may again appeal to Theorem \ref{Tgeneral} to conclude that $\sigma_1(\tilde{\Sigma}_n) = 1$.  
\end{proof}

Corollary \ref{CZ} gives more evidence to the statement in \cite{Z} that $\Sigma_n$ are likely congruent to a family of surfaces $M_n$ with $\sigma_1(M_n)=1$ previously constructed by Fraser and Schoen (Theorem 5.16, \cite{FS}).

%
%

\end{document}